\documentclass[11pt]{amsart}
\usepackage{amssymb,amscd,  latexsym, graphicx, mathrsfs, enumerate}

\newcommand{\nc}{\newcommand}
\numberwithin{equation}{section}
\newtheorem{theorem}{Theorem}[section]
\newtheorem{prop}[theorem]{Proposition}
\newtheorem{importnota}[theorem]{Important Notation}
\newtheorem{prblm}[theorem]{Problem}
\newtheorem{notation}[theorem]{Notation}
\newtheorem{caution}[theorem]{Caution}
\newtheorem{remark}[theorem]{Remark}
\newtheorem{lemma}[theorem]{Lemma}
\newtheorem{construction}[theorem]{Construction}
\newtheorem{corollary}[theorem]{Corollary}
\newtheorem{example}[theorem]{Example}
\newtheorem{conclusion}[theorem]{Conclusion}
\newtheorem{triviality}[theorem]{Triviality}
\newtheorem{proto}[theorem]{Prototype Quasifibration}
\newtheorem{cauex}[theorem]{Cautionary Example}
\newtheorem{propositiondef}[theorem]{Proposition-Definition}
\newtheorem{subth}{Nuisance}[theorem]
\newtheorem{ssubth}{ }[subth]
\newtheorem{conjecture}[theorem]{Conjecture}
\newtheorem{sidest}[theorem]{Side Story}
\newtheorem{miniexample}[theorem]{Example}
\theoremstyle{definition}
\newtheorem{defin}[theorem]{Definition}

\nc\tri[1]{\begin{triviality}}
\nc\side[1]{\begin{sidest}}
\nc\conj[1]{\begin{conjecture}}
\nc\prodef[1]{\begin{propositiondef}}
\nc\prt[1]{\begin{proto}}
\nc\lem[1]{\begin{lemma}}
\nc\sblm[1]{\begin{sublemma}}
\nc\pro[1]{\begin{prop}}
\nc\thm[1]{\begin{theorem}}
\nc\cor[1]{\begin{corollary}}
\nc\dfn[1]{\begin{defin}}
\nc\sthm[1]{\begin{subth}}
\nc\exm[1]{\begin{example}}
\nc\miniexm[1]{\begin{miniexample}}
\nc\plm[1]{\begin{prblm}}
\nc\rmk[1]{\begin{remark}}
\nc\subrmk[1]{\begin{subremark}}
\nc\ntn[1]{\begin{notation}}
\nc\cau[1]{\begin{caution}}
\nc\imn[1]{\begin{importnota}}
\nc\cax[1]{\begin{cauex}}
\nc\con[1]{\begin{construction}}
\nc\ssthm[1]{\begin{ssubth}}
\nc\cnc[1]{\begin{conclusion}}
\nc\elem{\end{lemma}}
\nc\esblm{\end{sublemma}}
\nc\eside{\end{sidest}}
\nc\econj{\end{conjecture}}
\nc\eprodef{\end{propositiondef}}
\nc\eprt{\end{proto}}
\nc\ethm{\end{theorem}}
\nc\ecor{\end{corollary}}
\nc\edfn{\end{defin}}
\nc\esthm{\end{subth}}
\nc\epro{\end{prop}}
\nc\etri{\end{triviality}}
\nc\eexm{\end{example}}
\nc\eminiexm{\end{miniexample}}
\nc\ermk{\end{remark}}
\nc\subermk{\end{subremark}}
\nc\eplm{\end{prblm}}
\nc\ecau{\end{caution}}
\nc\ecax{\end{cauex}}
\nc\eimn{\end{importnota}}
\nc\entn{\end{notation}}
\nc\econ{\end{construction}}
\nc\ecnc{\end{conclusion}}
\nc\essthm{\end{ssubth}}

\newcommand{\C}{\mathbb{C}}
\newcommand{\R}{\mathbb{R}}
\newcommand{\Q}{\mathbb{Q}}

\newcommand{\Z}{\mathbb{Z}}

\newcommand{\A}{\mathbb{A}}

\newcommand{\G}{\Gamma}

\newcommand{\ds}{\displaystyle}

\newcommand{\tr}{{\rm Tr}}

\newcommand{\lra}{\longrightarrow}

\newcommand{\bs}{\backslash}

\renewcommand{\Bbb}{\mathbb}
\makeatletter
\def\@xfootnote[#1]{%
  \protected@xdef\@thefnmark{#1}%
  \@footnotemark\@footnotetext}
\makeatother

\title[Miyawaki type lift for $GSpin(2,10)$]
{Miyawaki type lift for $GSpin(2,10)$}
\author{Henry H. Kim and Takuya Yamauchi}
\thanks{The first author is partially supported by NSERC. The second author
is partially supported by JSPS Grant-in-Aid for Scientific Research (C) No.15K04787}
\subjclass[2010]{}
\address{Henry H. Kim \\
Department of mathematics \\
 University of Toronto \\
Toronto, Ontario M5S 2E4, CANADA \\
and Korea Institute for Advanced Study, Seoul, KOREA}
\email{henrykim@math.toronto.edu}

\address{Takuya Yamauchi \\
Department of mathematics, Faculty of Education\\
Kagoshima University\\
Korimoto 1-20-6 Kagoshima 890-0065, JAPAN}
\email{yamauchi@edu.kagoshima-u.ac.jp}

\begin{document}
\begin{abstract}
Let $\frak T_2$ (resp. $\frak T$) be the Hermitian symmetric domain of $Spin(2,10)$ (resp. $E_{7,3}$).
In the previous work \cite{KY}, we constructed holomorphic cusp forms on $\frak T$ from elliptic cusp forms with respect to $SL_2(\Z)$.
By using such cusp forms we construct holomorphic cusp forms on $\frak T_2$ which are similar to Miyawaki lift in symplectic groups
established by T. Ikeda \cite{Ik1}. 
\end{abstract}
\maketitle

\section{Introduction}
Let $\A=\A_\Q$ be the ring of adeles of $\Q$. 
For a reductive group $G$ over $\Q$ of higher rank with Hermitian symmetric domain $\frak D$, it is important to construct
cuspidal representations of $G(\A)$ which give rise to holomorphic cusp forms on $\frak D$. In general it would be difficult to construct cusp forms directly. One way is to use Langlands functoriality, namely, consider another smaller group $H$ with an $L$-group homomorphism $r: {}^L H\lra {}^L G$, and then Langlands functoriality predicts a functorial lift from automorphic representations of $H(\A)$ to those of $G(\A)$. Some of cases are established by using the trace formula or the theta lift. These are very powerful tools, but the former never gives any explicit construction for classical forms and the latter can be made explicit with a careful choice of test functions, but it usually gives rise to automorphic representations which are generic, away from holomorphic forms, otherwise 
we need to consider a non-trivial level.

Contrary to these methods, Ikeda \cite{Ik} gave an explicit construction of cusp forms for the symplectic group $Sp_{n}$ (with $\Q$-rank $n$) from elliptic cusp forms of $GL_2(\A)$ with respect to $SL_2(\Z)$. Such a cusp form is called Ikeda lift.
In \cite{Ik1} Ikeda studied an integral similar to (\ref{integral}) below, obtained by substituting  the role of Eisenstein series in the
usual pullback formula with the Ikeda lift. Then under the assumption of nonvanishing of the integral, he showed that 
it gives rise to an essentially new cusp form for symplectic groups which is called Miyawaki lift.
The existence of both lifts is compatible with the conjectural Arthur's multiplicity formula which would be a theorem
soon \cite{Art}.

In this paper we pursue an analogue of Miyawaki type lift for $GSpin(2,10)$ by using our previous work \cite{KY}.
We now explain the main theorem. We refer the next section for several notations which appear below (or Section 2 of \cite{KY}).

Let $G=E_{7,3}$ and $G'=GSpin(2,10)$, which split at every prime $p$.
Let $\frak T_2$ (resp. $\frak T$) be the Hermitian symmetric domain of $PGSpin(2,10)(\R)^0$ (resp. $E_{7,3}(\R)$).
Any elements of $\frak T$ and $\frak T_2$ are described in terms of Cayley numbers $\frak C_\C$  and we can write $g\in \frak T$ as
$g=\begin{pmatrix} Z&w\\{}^t\overline{w}&\tau\end{pmatrix}$ with $Z\in \frak T_2$, $w\in \frak C^2_\C$, and
$\tau\in \mathbb{H}=\{z\in\C\ |\ {\rm Im} z>0\}$.
Let $S_{2k}(SL_2(\Z))$ be the space of elliptic cusp forms of weight $2k\geq 12$ with respect to $SL_2(\Z)$.
For each normalized Hecke eigenform $f=\ds\sum_{n=1}^\infty c(n)q^n,\ q=\exp(2\pi\tau \sqrt{-1}),\ \tau\in \mathbb{H}$, in $S_{2k}(SL_2(\Z))$, 
let $F_f$ be the Ikeda type lift on $\frak T$ of $f$ which was constructed in \cite{KY}.
This is a Hecke eigen cusp form of weight $2k+8$ with respect to $G(\Z)$.

For a normalized Hecke eigenform $h\in S_{2k+8}(SL_2(\Z))$, consider the integral
\begin{equation}\label{integral}
\mathcal{F}_{f,h}(Z)=\int_{SL_2(\Z)\backslash \Bbb H} F_f\begin{pmatrix} Z&0\\0&\tau\end{pmatrix} \overline{h(\tau)}
({\rm Im} \tau)^{2k+6}\, d\tau.
\end{equation}
Note that $\mathcal{F}_{f,h}(Z)$ is a cusp form (possibly zero) of weight $2k+8$ with respect to $Spin(2,10)(\Z)$.

For each prime $p$, let $\{\alpha_p, \alpha_p^{-1}\}$ and $\{\beta_p, \beta_p^{-1}\}$ be the Satake parameters of $f, h$ at $p$, resp.
Let $\pi_f, \pi_h$ be the cuspidal representations attached to $f$ and $h$ resp., and let $L(s,\pi_f)$, $L(s,\pi_h)$ be 
their automorphic $L$-functions. 

For a technical reason, we assume the Langlands functorial transfer of automorphic representations of $PGSpin(2,10)(\Bbb A)$ to $GL_{12}(\Bbb A)$:
Namely, given a cuspidal representation of $PGSpin(2,10)(\Bbb A)$ which is unramified at every prime $p$, there exists an automorphic representation of $GL_{12}(\Bbb A)$ which is unramified at every prime $p$, and their Satake parameters correspond under the $L$-group homomorphism ${}^L GSpin(2,10)=GSO(12,\C)\hookrightarrow GL_{12}(\C)$. The transfer is a composition of two transfers:
The transfer of automorphic representations of $PGSpin(2,10)(\Bbb A)$ to the split group $PGSpin(12,\Bbb A)$ is the Jacquet-Langlands correspondence.
Since $PGSpin(12)=PGSO(12)$, we can consider automorphic representations of $PGSO(12,\Bbb A)$ as automorphic representations of $SO(12,\Bbb A)$ with the trivial central character. The transfer of automorphic representations of $SO(12,\Bbb A)$ to $GL_{12}(\Bbb A)$ is now complete by Arthur \cite{Art}.

We prove

\begin{theorem} \label{main} Assume that $\mathcal{F}_{f,h}$ is not identically zero. Assume also the existence of the functorial transfer from 
$PGSpin(2,10)(\Bbb A)$ to $GL_{12}(\Bbb A)$.
Then 
\begin{enumerate}
\item  The cusp form $\mathcal{F}_{f,h}$ is a Hecke eigenform, and hence gives rise to a cuspidal representation $\Pi_{f,h}$ of $G'(\A)$ with the trivial central character, which is unramified at every prime $p$.
\item Let
$\Pi_{f,h}=\Pi_\infty\otimes \otimes_p' \Pi_p$. For each prime $p$, the Satake parameter of $\Pi_p$ is given by 
$$\{ (\beta_p\alpha_p)^{\pm 1}, \ (\beta_p\alpha_p^{-1})^{\pm 1} ,\, 1, 1, 
p^{\pm 1},\ p^{\pm 2},\ p^{\pm 3}\}
$$
\item The degree 12 standard L-function of the cuspidal representation $\Pi_{f,h}$  is given by  
$$L(s,\Pi_{f,h})=L(s,\pi_f\times \pi_h)\zeta(s)^2 \zeta(s\pm 1)\zeta(s\pm 2)\zeta(s\pm 3),
$$  
where the first $L$-function is the Rankin-Selberg $L$-function.
\item The transfer of $\Pi_{f,h}$ to $GL_{12}(\Bbb A)$ is $(\pi_f\boxtimes\pi_h)\boxplus 1_{GL_7}\boxplus 1$, where $1_{GL_7}$ is the trivial representation of $GL_7(\Bbb A)$.
\end{enumerate}
\end{theorem} 
We first show (Proposition \ref{key}) that the Satake parameter of $\Pi_p$ is given by 

\begin{eqnarray*}
&& (I)_p:\ \{\varepsilon_p(\beta_p\alpha_p)^{\pm 1},\ \varepsilon_p(\beta_p\alpha_p^{-1})^{\pm 1},\ b_p^{\pm 1},\ (b_p p)^{\pm 1},\ (b_p p^2)^{\pm 1},\ (b_p p^3)^{\pm 1}\},
\quad {\rm or}   \\
&&(II)_p:\ \{ \varepsilon_p(\beta_p\alpha_p)^{\pm 1}, \ \varepsilon_p(\beta_p\alpha_p^{-1})^{\pm 1}, 
\varepsilon_p(\beta_p\alpha_p^{-1}p)^{\pm 1},\ \varepsilon_p(\beta_p\alpha_p^{-1}p^2)^{\pm 1},\ 
\varepsilon_p(\beta_p\alpha_p^{-1}p^3)^{\pm 1},\ 
\varepsilon_p(\beta_p\alpha_p^{-1}p^4)^{\pm 1}\},
\end{eqnarray*}
where $\varepsilon_p\in \{\pm 1\}$ and $b_p\in\Bbb C^\times$. Using the functorial transfer, in Section 6, we prove that only $(I)_p$ occurs, remove the sign ambiguity and $b_p=1$.

\begin{remark} If we take $h=E_{2k+8}$, the Eisenstein series of weight $2k+8$, the integral (\ref{integral}) still makes sense and defines a cusp form of weight $2k+8$ with respect to $Spin(2,10)(\Bbb Z)$. If $\mathcal F_{f,E_{2k+8}}$ is not zero, then it gives rise to a cuspidal representation $\Pi_{f,E_{2k+8}}$ of $GSpin(2,10)$, and the standard $L$-function of $\Pi_{f,E_{2k+8}}$ is 
$$L(s,\Pi_{f,E_{2k+8}})=L(s+\frac 12, \pi_f)L(s-\frac 12,\pi_f)\zeta(s)^2\zeta(s\pm 1)\zeta(s\pm 2)\zeta(s\pm 3).
$$
\end{remark}

\begin{remark} Here $\Pi_\infty$ is a holomorphic discrete series of the lowest weight $2k+8$.
Since $f$ and $h$ have different weights, they can never be equal. Therefore $L(s,\pi_f\times \pi_h)$ is entire.
\end{remark}

\begin{remark} Note that
${}^L Spin(2,10)=PGSO(12,\C)$, and $PGSO(12,\C)$ does not have a $12$-dimensional representation. 
The minimum dimension among of the algebraic irreducible representations of $PGSO(12,\C)$ is $66$ by 
Weyl's dimension formula and that is given by 
${\rm Ad}:PGSO(12,\C) \lra GL({\rm Lie}(PGSO(12,\C)) \cong GL_{66}(\C)$.  
Therefore, given a cuspidal representation $\pi$ of $Spin(2,10)$, we cannot define the degree 12 standard $L$-function of $\pi$.
However, ${}^L GSpin(2,10)=GSO(12,\C)$, and $GSO(12,\C)$ has a $12$-dimensional representation. Since $PGSpin(2,10)=PGSO(2,10)$,
our form $\Pi_{f,h}$ can be considered as a cuspidal representation of $GSpin(2,10)$ with the trivial central character.

This situation is similar to Siegel cusp forms. Given a Siegel cusp form $F$ on a degree 2 Siegel upper half plane, we need to consider a cuspidal representation $\pi_F$ of $GSp_4$, rather than
$Sp_4$ in order to define the degree 4 spin $L$-function.
\end{remark}

\begin{remark}  We give a conjectural Arthur parameter of $\Pi_{f,h}$:
Let $\phi_f, \phi_h: \mathcal L\lra SL_2(\C)$ be the hypothetical Langlands parameter attached to $f,h$, resp. 
We have the tensor product map $SL_2(\C)\times SL_2(\C)\lra SO_4(\C)$. [\cite{OV}, page 88. Use the identification $SL_2(\C)=Sp_1(\C)$, and we have a representation of $SL_2(\C)\times SL_2(\C)$ on $\C^2\otimes\C^2\simeq \C^4$. It defines a symmetric, non-degenerate bilinear form on $\C^4$.] Then
we have
$\phi_f\otimes\phi_h: \mathcal L\lra SO_4(\Bbb C).$
The distinguished unipotent orbit $(7,1)$ of $SO_8(\C)$ gives rise to a map
$SL_2(\C)\lra SO_8(\C)$. Hence it defines a map
$\phi_u: \mathcal L\times SL_2(\Bbb C)\lra SO(8,\Bbb C)$. Then consider
$$\phi=(\phi_h\otimes \phi_f)\oplus  \phi_u : \mathcal L\times SL_2(\C)\lra SO_4(\C)\times SO_8(\C)\subset GSO_{12}(\C).
$$
We expect that $\phi$ parametrizes $\Pi_{f,h}$.
\end{remark}

This paper is organized as follows. In Section 2, we recall several facts about the Hermitian symmetric domain of $Spin(2,10)$ or $PGSpin(2,10)=PGSO(2,10)$,
and holomorphic modular forms on it.
In Section 3, we recall our previous work \cite{KY}, In Sections 5 and 6, following Ikeda \cite{Ik1}, 
we study the integral expression (\ref{integral}) for $\mathcal F_{f,h}$, which gives rise to a cusp form on $\frak T_2$.
We carry out the essentially same method but we have to rely on roots to describe some double coset space related to this method.
The calculation of the double cosets will be devoted in Section 4. 
In Section 7, we compute $\mathcal F_{f,h}$ explicitly using two kinds of Fourier-Jacobi expansions, and indicate that it is most likely nonvanishing.

\medskip

\textbf{Acknowledgments.} We would like to thank T. Ibukiyama, T. Ikeda, H. Katsurada, R. Lawther, A. Luzgarev, T. Moriyama and T. Sugano for their valuable comments.
In particular Lawther sent very detailed notes \cite{L1} to us on the double coset decomposition in Section \ref{double} 
and Luzgarev sent a mathematica code to the second author. 
Ikeda also pointed out several mistakes in previous version. 
Without their help, this paper could not have been finished.  

\section{Preliminaries}

\subsection{Cayley numbers and the exceptional domain}\label{T2}
In this section we refer Section 2 of \cite{KY}.
For any field $K$ whose characteristic is different from $2$ and $3$,
the Cayley numbers $\frak C_K$ over $K$ is an eight-dimensional vector space over $K$ with basis
$\{e_0=1,e_1,e_2,e_3,e_4,e_5,e_6,e_7\}$ satisfying the following rules for multiplication:
\begin{enumerate}
\item  $xe_0=e_0x=x$ for all $x\in\frak C_K$,
\item $e^2_i=-e_0$ for $i=1,\ldots, 7$,
\item $e_i(e_{i+1}e_{i+3})=(e_ie_{i+1})e_{i+3}=-e_0$ for any $i$ (mod 7).
\end{enumerate}
For each $x=\ds\sum_{i=0}^7x_ie_i\in \frak C_K$,
the map $x\mapsto \bar{x}=x_0e_0-\ds\sum_{i=1}^7x_ie_i$ defines an anti-involution on $\frak C_K$.
The trace and the norm on $\frak C_K$ are defined by
$$\tr(x):=x+\bar{x}=2x_0,\ N(x):=x\bar{x}=\sum_{i=0}^7x^2_i.
$$
The Cayley numbers $\frak C_K$ is neither commutative nor associative. We denote by $\frak o$, the integral Cayley numbers
which is a $\Z$-submodule of $\frak C_K$ given by the following basis:
$$\alpha_0=e_0,\ \alpha_1=e_1,\ \alpha_2=e_2,\ \alpha_3=-e_4,\
\alpha_4=\frac{1}{2}(e_1+e_2+e_3-e_4),\ \alpha_5=\frac{1}{2}(-e_0-e_1-e_4+e_5),$$
$$\alpha_6=\frac{1}{2}(-e_0+e_1-e_2+e_6),\ \alpha_7=\frac{1}{2}(-e_0+e_2+e_4+e_7).$$
It is known that $\frak o$ is stable under the operations of the anti-involution, multiplication, and
addition. Further we have $\tr(x),\ N(x)\in \Z$ if $x\in \frak o$. By using this integral structure,
for any $\Z$-algebra $R$, one can consider $\frak C_R=\frak o\otimes_\Z R$.

Let $\frak J_K$ be the exceptional Jordan algebra consisting of the element:
\begin{equation}\label{x}
X=(x_{ij})_{1\le i,j\le 3}=\left(\begin{array}{ccc}
a& x & y \\
\bar{x} & b & z \\
\bar{y}& \bar{z} & c
\end{array}\right),
\end{equation}
where $a,b,c\in Ke_0=K$ and $x,y,z\in \frak C_K$.

By using integral Cayley numbers, we define a lattice
$$\frak{J}(\Z):=\{X=(x_{ij})\in \frak J_\Q\ |\ \text{$x_{ii}\in \Bbb Z$, and $x_{ij}\in \frak o$ for $i\ne j$} \},
$$
and put $\frak{J}(R)=\frak{J}(\Z)\otimes_\Z R$ for any $\Z$-algebra $R$.

We define
$$R_3(K)=\{X\in \frak J_K\ |\ \det(X)\not=0 \}
$$
and define the set $R^+_3(K)$ consisting of squares of elements in $R_3(K)$.
It is known that $R^+_3(\R)$ is an open, convex cone in $\frak J_\R$.
We denote by $\overline{R^+_3(\R)}$ the closure of $R^+_3(\R)$ in $\frak J_\R\simeq \R^{27}$
with respect to Euclidean topology. For any subring $A$ of $\R$, set
$$\frak J(A)_+:=\frak J(A)\cap R^+_3(\R),\quad \frak J(A)_{\ge 0}:=\frak J(A)\cap \overline{R^+_3(\R)}.
$$
We define the exceptional domain as follows:
$$\frak T:=\{Z=X+Y\sqrt{-1}\in \frak J_\C\ |\ X,Y\in \frak J_\R,\ Y\in R^+_3(\R)\}$$
which is a complex analytic subspace of $\C^{27}$ .

Let $G$ be  the exceptional Lie group of type $E_{7,3}$ over $\Q$ which acts on $\frak T$.
Then $G(\Bbb R)$ is of real rank 3 (cf. \cite{B}). In loc.cit. Baily constructed an integral model $\mathcal G_\Z$ of $G$ over ${\rm Spec}\hspace{0.5mm}\Z$ 
and it follows from this with Proposition 1.1 of \cite{Gro} that $G(\Bbb Q_p)$ is a split group of type $E_7$ for any prime $p$.

The Satake diagram of $E_{7,3}$ is
\begin{align*}
{\text{o}}_{\beta_1}\text{------}{\bullet}_{\beta_3}{%
\text{------}}&{\bullet}_{\beta_4}{\text{------}}{\bullet}_{\beta_5}{\text{------}}{\text{o}}_{\beta_6}{\text{------}}{%
\text{o}}_{\beta_7} \\
&\Big\vert \\
&{\bullet}_{\beta_2}
\end{align*}

The $\Bbb Q$-root system is of type $C_3$, and the extended Dynkin diagram of $C_3$ is

\begin{align*}
{\text{o}}_{\lambda_0}\Longrightarrow{\text{o}}_{\lambda_1}{%
\text{------}}&{\text{o}}_{\lambda_2}{\Longleftarrow}{\text{o}}_{\lambda_3},
\end{align*}
where $\lambda_1$ corresponds to $\beta_1$, $\lambda_2$ to $\beta_6$, $\lambda_3$ to $\beta_7$, and $-\lambda_0$ is the 
maximal root in $C_3$. Here $\lambda_1,\lambda_2$ have multiplicity 8, and $\lambda_3$ has multiplicity 1.

Let $G_1=SL_2$, $G_2=Spin(2,10)$.\footnote[*]{Since we are not dealing with the exceptional group of type $G_2$, we hope that our notation will not cause confusion.}
Then $(G_1,G_2)$ is a dual pair inside $G=E_{7,3}$ (cf. \cite{Du}). 
They are given as follows: If we remove the root $\lambda_1$ in the extended Dynkin diagram, the remaining diagram is an almost direct product $G_1G_2$. More precisely, let $\theta = h_{\lambda_0}(-1)$.
Then $\theta$ is an involution whose centralizer $H=C_{E_7}(\theta)$ as an algebraic group is the almost direct product $G_1G_2$.  Then
$G_1\cap G_2=Z=\{(h_{\lambda_0}(-1), h_{\lambda_0}(-1)\}\simeq \{\pm1\}$.   
Since $G_1$ and $G_2$ are simply connected algebraic groups, one has the following exact sequence 
$$1\lra \mu_2(k)\lra  G_1(k)\times G_2(k)\lra H(k)\lra H^1({\rm Gal}(\overline{k}/k),k^\times)=k^\times/(k^\times)^2\lra 1$$
for any local field $k$ of characteristic zero. This means that $H(k)$ is strictly bigger that $G_1(k)G_2(k)\subset E_7(k)$. 
Furthermore the 2 to 1 isogeny $G_1\times G_2\lra H$ induces 
a natural inclusion $X^\ast(T_H)\hookrightarrow X^\ast(T_{G_1})\times X^\ast(T_{G_2})$ of index 2 where 
$X^\ast(T)$ stands for the character group of a torus $T$. 

We remark that $G_2(k)$ is a split group for any $p$-adic field $k$. The $\Bbb Q$-root system of $G_2$ is of type $C_2$. It is the group $\mathfrak L_2$ in \cite{B}, page 528, and it acts on the boundary component $\frak T_2$ below.

To end this section, we remark on an explicit integral model of $G_2=Spin(2,10)$. 
Since ${G_2}_\Q\subset \mathcal G_\Z$, one can define an integral model $\mathcal G_2$ of $G_2$ 
as the Zariski closure of ${G_2}_\Q$ in $G_\Z$. It follows from Proposition 1.1 of \cite{Gro} again that 
$\mathcal G_2$ is a smooth model over $\Z$. 
Then we have $G_2(\Z)=Spin(2,10)(\Q)\cap G(\Z)$. 

We can construct an explicit integral model of $G_2$ up to $\Q$-isomorphism as follows.   
There is a natural surjective map  $\iota: {G}_2\lra G_2/\{h_{\lambda_0}(-1)\}=SO(2,10)$ with kernel $\mu_2$, where $SO(2,10)$ is the 
special orthogonal group we want to define explicitly. 
Since $G_2(\Q_p)$ splits, so does $SO(2,10)(\Q_p)$ for any prime $p$. By Hasse principle, there exists a unique 
$\Q$-isomorphism class of $SO(2,10)$ which splits everywhere (Theorem 4.1.2 of \cite{Kit}). 
On the other hand  the quadratic space $V=H\perp H\perp (-E_8)$ where 
$E_8$ is the quadratic form given by the Cartan matrix of the exceptional Lie algebra of type $E_8$ and $H$ is the usual hyperbolic space, defines 
a special orthogonal group $SO(V)$ with the signature $(2,10)$ which splits at any prime $p$. Hence we have 
$SO(V)\simeq SO(2,10)$ over $\Q$.  
Then $Spin(2,10)$ is 
defined as the double cover of $SO(V)$ via the isomorphism $SO(V)\simeq SO(2,10)$ as above.  

\subsection{Hermitian symmetric domain for $GSpin(2,10)$}

Define $\frak{J}_2(R)$ as the set of all matrices of forms
$$X=\left(\begin{array}{cc}
a& x  \\
\bar{x} & b
\end{array}\right),\ a,b\in R,\ x\in \frak C_R.$$
We define the inner product on $\frak{J}_2(R)\times \frak{J}_2(R)$
by $(X,Y):=\frac 12\tr(XY+YX)$.
For any such $X$, we define $\det(X):=ab-N(x).$
For $X$ as above, $r\in R$, and $\xi=\left(\begin{array}{c}
\xi_1 \\
\xi_2
\end{array}\right),\ \xi_i\in \frak C_R\ (i=1,2)$, we have
$\left(\begin{array}{cc}
X & \xi \\
{}^t\bar{\xi} & r
\end{array}\right)\in \frak{J}(R)$.
Define
$$\frak J_2(A)_{+}=\Bigg\{\left(\begin{array}{cc}
a& x  \\
\bar{x} & b
\end{array}\right)\in\frak J_2(A) \ \Bigg|\ a,b\in A\cap \R_{>0},\ ab-N(x)>0  \Bigg\},$$
and
$$\frak J_2(A)_{\ge 0}=\Bigg\{\left(\begin{array}{cc}
a& x  \\
\bar{x} & b
\end{array}\right)\in \frak J_2(A)\ \Bigg|\ a,b\in A\cap \R_{\ge 0},\ ab-N(x)\ge 0  \Bigg\}.$$

We also define
$$\frak T_2:=\{X+Y\sqrt{-1}\in \frak J_2(\C)\ |\ X,Y\in \frak J_2(\R),\ Y\in \frak J_2(\R)_+\}.
$$

It is well-known that $\frak T_2$ is the Hermitian symmetric domain for $G_2(\R)$ which is a tube domain of
type (IV).
Since $Spin(2,10)(\R)/\{\pm 1\}\simeq SO(2,10)(\R)$, where $\{\pm 1\}$ is a subgroup in the center of $Spin(2,10)(\R)$,  
$\frak T_2$ is also the symmetric domain for $SO(2,10)(\R)$ (See Section 6 of Appendix in \cite{Sat}).
For us, it is more convenient to consider $\tilde G=PGSO(2,10)=PGSpin(2,10)$. In this case, $\frak T_2$ is also the symmetric domain for $PGSO(2,10)(\R)^0$.
Then modular forms on $\frak T_2$ can be considered as automorphic forms on $GSpin(2,10)(\A_\Q)$ with the trivial central character. 

\subsection{Modular forms on $\frak T_2$}
Recall the integral model of $G_2=Spin(2,10)$ over $\Z$ from Section \ref{T2}. 
Then one can define the arithmetic group $\Gamma_2=G_2(\Z)$ of ``level one". 
In \cite{EK}, Eie and Krieg considered an arithmetic subgroup $\Gamma'\subset \Gamma_2$, generated by the following. For $Z\in \frak T_2$, let $Z=\begin{pmatrix} z_1&w\\ \bar w&z_2\end{pmatrix}$,
where $z_1,z_2\in \Bbb H$, and $w=x+y\sqrt{-1}$ with $x,y\in \frak C_\Bbb R$, and $\bar w=\bar x+\bar y\sqrt{-1}$. Let $det(Z)=z_1z_2-w\bar w$:
\begin{enumerate}
 \item $p_B: Z\mapsto Z+B$, $B\in\frak J_2(\Bbb Z)$;
 \item $t_U: Z\mapsto {}^t \bar UZU$, $U=\begin{pmatrix} 0&1\\-1&0\end{pmatrix}$ or $U=\begin{pmatrix} 1&u\\0&1\end{pmatrix}$ for $u\in\frak o$;
 \item $\iota: Z\mapsto -Z^{-1}$, where $Z^{-1}=\frac 1{\det(Z)} \begin{pmatrix} z_2&-w\\-\bar w&z_1\end{pmatrix}$.
\end{enumerate}

If we consider $\Gamma'$ as a subgroup of $G(\Z)$, $p_B$ is the element $p_{B'}$ in \cite{KY} with $B'=\begin{pmatrix} 0&0\\0&B\end{pmatrix}$ and $B\in \frak J_2(\Bbb Z)$; and $\iota=\iota_{e_2}\iota_{e_3}$ in \cite{KY}. Also $t_U=m_{u e_{23}}\in M(\Z)$ in \cite{KY} for $U=\begin{pmatrix} 1&u\\0&1\end{pmatrix}$. If $U=\begin{pmatrix} 0&1\\-1&0\end{pmatrix}$, $t_U=m_{e_{23}}m_{-e_{32}}m_{e_{23}}$.

It is likely that $\Gamma'=\Gamma_2$, but we have not shown it yet.

For any $g\in G'(\R)$ and $Z\in \frak T_2$, one can define
a holomorphic automorphic factor $j(g,Z)\in \C$ which satisfies the cocycle condition. 
More explicitly, $j(p_B,Z)=j(t_U,Z)=1$ and $j(\iota,Z)=\det(Z)$.

Let $F$ be a holomorphic function on $\frak T_2$ which for some integer $k>0$ satisfies
$$
F(\gamma Z)=F(Z) j(\gamma,Z)^k,\quad Z\in \frak T_2,\, \gamma\in\Gamma_2.
$$
Then $F$ is called a modular form on $\frak T_2$ of weight $k$ with respect to $\Gamma_2$. 
For example, $F$ satisfies
$$F(Z+B)=F(Z),\quad F({}^t \bar UZU)=F(Z),\quad F(-Z^{-1})=\det(Z)^k F(Z),
$$
for $B\in\frak J_2(\Bbb Z)$, and $U=\begin{pmatrix} 1&u\\0&1\end{pmatrix}$ for $u\in\frak o$.

We denote by $\mathcal{M}_k(\Gamma_2)$ the space of such forms.
By Koecher principle, we do not need the holomorphy at the cusps.
For a holomorphic function $F: \frak T_2\lra \C$, consider, for $\tau\in\Bbb H$,
$$\Phi F(\tau)=\lim_{y\to \infty} F\begin{pmatrix} \tau &*\\ * & iy\end{pmatrix}.
$$
If $\Phi F=0$, $F$ is called a cusp form. Let $\mathcal{S}_k(\Gamma_2)$ be the space of cusp forms of weight $k$ with respect to $\G_2$.


By  \cite{PR}, Theorem 7.12, the strong approximation theorem holds with respect to $S=\{\infty\}$, namely,
$G'(\A)=G'(\Q)G'(\R) G'(\widehat{\Z})$, and $G_2(\Z)=G'(\Q)\cap G'(\R) G'(\widehat{\Z})$.
Hence one can associate  a Hecke eigen cusp form in $S_k(\Gamma_2)$ with
an automorphic form on $G'(\A)$ which is fixed by $G'(\widehat{\Z})$, and then we obtain a cuspidal
automorphic representation of $G'(\A)$ with the trivial central character.

\section{Ikeda type lift for $E_{7,3}$}
In this section we recall the Ikeda type construction for $E_{7,3}$ in \cite{KY}.
Let $P=MN$ be the Siegel parabolic subgroup of ${G}$ where the derived group $M^D=[M,M]$ of
the Levi subgroup $M$ is of type $E_6$. Let $\nu:M\lra GL_1$ be the similitude character (see Section 2 of \cite{KY}) and
it can be naturally extended to $P$.
Let $\Gamma={G}(\Z)$ be the arithmetic subgroup defined by Baily in \cite{B} which is
constructed by using the integral Cayley numbers $\frak o$.
For a positive integer $k\ge 6$, we constructed in \cite{KY} a non-zero Hecke eigen cusp form $F_f(Z)$ in $S_{2k+8}(\G)$
from a Hecke eigen cusp form $f=\ds\sum_{n\ge 1}c(n)q^n\in S_{2k}(SL_2(\Z))$:
For a positive integer $k\ge 6$, let $E_{2k+8}$ be the Siegel Eisenstein series on $\frak T$ of
weight $2k+8$ with respect to $\Gamma$. Then it has the Fourier expansion of form
\begin{eqnarray*}
E_{2k+8}(Z) &=& \sum_{T\in \frak J(\Z)_+} a_{2k+8}(T) \exp(2\pi\sqrt{-1}(T,Z)),\ Z\in\frak T,\\
a_{2k+8}(T) &=& C_{2k+8}\det(T)^{\frac{2k-1}{2}}\prod_{p|\det(T)} \widetilde{f}^p_T(p^{\frac{2k-1}{2}}),
\end{eqnarray*}
where $C_{2k+8}=2^{15}\displaystyle\prod_{n=0}^2 \frac {2k+8-4n}{B_{2k+8-4n}}$, and
$\widetilde{f}^p_T(X)$ is a Laurent polynomial over $\Q$ in $X$ which is depending only on $T$ and $p$.

Let $S_{2k}(SL_2(\Z))$ be the space of elliptic cusp forms of weight $2k\geq 12$ with respect to $SL_2(\Z)$.
For each normalized Hecke eigenform $f=\ds\sum_{n=1}^\infty c(n)q^n,\ q=\exp(2\pi \sqrt{-1}\tau),\ \tau\in \mathbb{H}$ in $S_{2k}(SL_2(\Z))$ and each rational prime $p$,
we define the Satake $p$-parameter $\{\alpha_p,\alpha_p^{-1}\}$ by $c(p)=p^{\frac{2k-1}{2}}(\alpha_p+\alpha^{-1}_p)$.
For such $f$, consider the following formal series on $\frak T$:
$$F_f(Z)=\sum_{T\in \frak J(\Z)_+} A(T)\exp(2\pi\sqrt{-1}(T,Z)),\ Z\in \frak T, \quad
A(T)=\det(T)^{\frac{2k-1}{2}} \prod_{p|\det(T)} \widetilde{f}_T^p(\alpha_p).
$$
Then we showed
\begin{theorem}\cite{KY}  The function $F_f(Z)$ is a non-zero Hecke eigen cusp form on $\frak T$ of weight $2k+8$ with respect to
$\G$.
\end{theorem}

We call $F_f$ the Ikeda type lift of $f$.
Then $F=F_f$ gives rise to a cuspidal automorphic representation
$\pi_F=\pi_\infty\otimes \otimes'_p\pi_p$
of $\bold G(\A)$.
Then $\pi_\infty$ is a holomorphic discrete series of the lowest weight $2k+8$ associated to $-(2k+8)\varpi_7$ in the notation of \cite{Bour} (cf. \cite{knapp}, page 158).
For each prime $p$, $\pi_p$ is unramified. In fact, $\pi_p$ turns out to be a degenerate principal series
\begin{equation}\label{deg}
\pi_p\simeq {\rm Ind}_{\bold P(\Q_p)}^{\bold{G}(\Q_p)}\: |\nu(g)|^{2s_p},
\end{equation}
where
$p^{s_p}=\alpha_p$ (see Section 11 of \cite{KY}).
Let $L(s,\pi_f)=\prod_p (1-\alpha_p p^{-s})(1-\alpha_p^{-1} p^{-s})$ be the
automorphic $L$-function of the cuspidal representation $\pi_f$ attached to $f$. Then

\begin{theorem}\cite{KY}  The degree $56$ standard L-function $L(s,\pi_F,St)$ of $\pi_F$ is given by
$$L(s,\pi_F,St)=L(s,{\rm Sym}^3 \pi_f)L(s,\pi_f)^2 \prod_{i=1}^4 L(s\pm i,\pi_f)^2 \prod_{i=5}^8 L(s\pm i, \pi_f),
$$
where $L(s,{\rm Sym}^3 \pi_f)$ is the symmetric cube $L$-function.
\end{theorem}

\section{Double Coset Decomposition}\label{double}

In order to prove Theorem \ref{main}, following \cite{Ik1}, we need to compute
a suitable representatives of the double coset space over a $p$-adic field related to the unwinding method.

This section is mainly due to R. Lawther. We thank him for a very detailed note \cite{L1}.
He gave an explicit double coset space related to what we need, but he worked over an algebraically closed field because he relied on the results in
\cite{L}. In what follows we modify his argument so that it would work over any $p$-adic field in our case.

Let $p$ be any rational prime, and $G$ to be a simply-connected algebraic group of type $E_7$ over a $p$-adic field $k$, and for simplicity, let $G=G(k)$,
$G_1=G_1(k)$, and $G_2=G_2(k)$.

Let $T$ to be a fixed maximal torus of $G$. Let $B$ be the standard Borel subgroup containing $T$.
Take roots with respect to $T$; let $\{\beta_1, \dots, \beta_7 \}$ be a simple root system, numbered as in Bourbaki \cite{Bour}.
Write roots of $E_7$ as strings of coefficients of simple roots, so that for example, the highest root is $2234321$.
Let $\Phi$ (resp. $\Phi^+$) be the set of all roots (resp. all positive roots). Let $\gamma_1=0112221$, and by adding $\gamma_1$, we get
the extended Dynkin diagram of $E_7$;
\begin{align*}
{\text{o}}_{\gamma_1}\cdots {\text{o}}_{\beta_1}\text{------}{\text{o}}_{\beta_3}{%
\text{------}}&{\text{o}}_{\beta_4}{\text{------}}{\text{o}}_{\beta_5}{\text{------}}{\text{o}}_{\beta_6}{\text{------}}{%
\text{o}}_{\beta_7} \\
&\Big\vert \\
&{\text{o}}_{\beta_2}
\end{align*}

In order to use Lawther's note \cite{L1}, we take a different centralizer from Section \ref{T2}: Let $\theta = h_7(-1)$.
Then $\theta$ is an involution whose centralizer $H=C_{E_7}(\theta)$ is of the form $A_1 D_6$. Explicitly, the roots whose root subgroups lie in $H$ are those whose $\beta_6$-coefficient is even. The simple roots of the $D_6$ are
$\gamma_1$, and $\gamma_2=\beta_1, \gamma_3=\beta_3, \gamma_4=\beta_4, \gamma_5=\beta_5, \gamma_6=\beta_2$, and that of the $A_1$ is $\beta_7$. Then 
$Z:=G_1(k)\cap G_2(k)=\{(h_{\beta_7}(-1),h_{\beta_7}(-1))\}\simeq \{\pm1\}$.
Note that $h_{\gamma_1}(-1)h_{\gamma_3}(-1)h_{\gamma_6}(-1)=h_{\beta_7}(-1)$ and
$H(k)$ contains $G_1G_2\simeq  G_1(k)\times G_2(k)/Z$.

We set $\gamma_6=\beta_6$ and $\gamma_7=\beta_7$.

\subsection{Double coset space}\label{double}

For each root $\alpha$ let $x_{\alpha}(c), c\in k$ be the corresponding root subgroup and put
$$n_{\alpha}=x_{\alpha}(1)x_{-\alpha}(-1)x_{\alpha}(1),\ y_{\alpha}=x_{\alpha}(1)n_{\alpha}x_{\alpha}(\frac{1}{2}),\ {\rm and}\
h_\alpha(c)=x_{\alpha}(c)x_{-\alpha}(-c^{-1})x_{\alpha}(c)n^{-1}_{\alpha},\ c\in k^{\times}.$$
Put $h_i=h_{\beta_i}$ for simplicity.

Let $P$ be the Siegel parabolic subgroup of $G$ corresponding to
$\{\beta_1, \dots, \beta_6\}$ which is of type $E_6 T_1 U_{27}$ over an algebraic closed field where $T_i$ denotes an $i$-dimensional torus, and $U_j$ is a unipotent group of dimension $j$.
For each element $g\in G(k)$, put $Q_g=g^{-1}P(k)g\cap H$.
Then we have the following lemma.
\begin{lemma}\label{P-double} The double coset space $P(k)\backslash G(k)/H$ is a finite set. For any $g\in G(k)$,
there exists $g'$ so that $P(k)gH=P(k)g'H$ and
$Q_{g'}$ coincides with
$Q_y$ for some $y\in \{1,n, y_{\beta_7} n, y_{\gamma_1} y_{\beta_7} n\}$, where $n = n_{\beta_6 + \beta_7}$.
\end{lemma}
To prove this lemma, we need more arguments which would be a lengthy calculation.
Let $C$ be any fixed complete system of the
representatives of the Weyl group $W=N/T(k)$.
\begin{lemma}\label{B-double} A complete system of representatives of the
double coset space $B(k)\backslash G(k)/H$ is a finite set and it consists of the elements of form
$$y_{\alpha_1}\cdots y_{\alpha_r}n',\ \alpha_1,\ldots,\alpha_r\in \Phi^+ ,\ 0\le r \le 7$$
where $\alpha_1,\ldots,\alpha_r$ are mutually orthogonal and $n'\in N=N_{G(k)}(T(k))$ runs over the set $\{1\}\cup \{n'\in C\ | \ {}^{n'}\theta:=n'\theta n'^{-1}\not=\theta \}$. Furthermore $\alpha_i({}^{n'}\theta)=-1$ for any $i$ ($1\le i \le r$).
\end{lemma}

\begin{proof} The proof is almost same as in Section 3 of \cite{L} but we have to take care of the base field because
the results in \cite{L} stated for which the base field is an algebraically closed field.

Put $S=\{g\theta(g)^{-1}\ |\ \theta(g):=\theta g\theta,\  g\in G(k)\}$. Define the action of $G(k)$ on $S$ by
$$g\ast s=gs\theta(g)^{-1}\ s\in S,\ g\in G(k).$$
Let us define a bijective map $$B(k)\backslash G(k)/H\lra \{O_{B(k)}(s)\ |\ s\in S\},\ BxK\mapsto O_{B(k)}(x\theta(x)^{-1})$$
where $O_{B(k)}(s)$ stands for the orbit of $s$ for $B(k)$ with respect to the action $\ast$ as above.
By Proposition 6.6 of \cite{HW}, $O_{B(k)}(s)\cap N\not=\emptyset$, hence there exists $b\in B(k)$ such that
$b\ast s\in N$. Since $\theta\in N$, if $s=x\theta(x)^{-1}$ we also have $(b\ast s) \theta=bx\theta(bx)^{-1}=:{}^{bx}\theta\in N$ which is an involution
(hence $({}^{bx}\theta)^2=1$).
Take another $b'\in B(k)$ so that $b'\ast s\in N$ if exists. Put $g=bb'^{-1}\in B(k)$. Then
${}^{bx}\theta$ is conjugate by $g$ to ${}^{b'x}\theta$. By using Bruhat decomposition,
${}^{bx}\theta$ is conjugate by an element of $T(k)$ to ${}^{b'x}\theta$.
Hence ${}^{bx}\theta$ is unique up to the conjugate by $T(k)$ and thereby we may denote such a $b$ by $b_x$ with the
dependence on $x$.
Summing up we have  an injective map
$$
B(k)\backslash G(k)/H\hookrightarrow \{n\in N\ |\ n^2=1\}/\stackrel{T}{\sim}, \quad BxH\mapsto {}^{b_xx}\theta
$$
where  $\stackrel{T}{\sim}$  stands for the equivalence relation of the conjugation by elements in $T$.  
We now describe the image of this map. Let $g={}^{b_xx}\theta\in N$ be an involution for some $x\in G(k)$
and $b_x\in B(k)$.
Then by the proof of Lemma 2 of \cite{L} (noting that $n_{-\alpha}=n^{-1}_\alpha=n_{\alpha}t$ for some $t\in T(k)$),
there exists $\theta'\in T(k)$ and $t\in T(\overline{k})$ ($t=t_2t_1$ for $t_1$ at line 3, p.119 of \cite{L} and
$t_2$ at line 11,p. loc.cit.) such that
$${}^{t}g=\theta'n_{\alpha_1}\cdots n_{\alpha_r},\ 0\le r\le 7$$
such that $\alpha_1,\ldots,\alpha_r\in \Phi^+$ are mutually orthogonal and they satisfy
$\alpha_i(\theta')=-1$.

We now descent $t$ to an element in $T(k)$.
Let $Z_{T}(\theta' n)$ be the centralizer of $\theta' n$ in $T$ as an algebraic group over $\overline{k}$.
Put $n=n_{\alpha_1}\cdots n_{\alpha_r}$.
It is easy to see that
$Z_{T}(\theta' n)$ is defined over $k$ and it is also a split torus.
We define  a one-cocycle on ${\rm Gal}(\overline{k}/k)$ takes the values in $Z_{T}(\theta' n)(\overline{k})$ by
$$\sigma\mapsto t(t^{-1})^\sigma.$$
Since $H^1({\rm Gal}(\overline{k}/k),Z_{T}(\theta' n)(\overline{k}))=1$ by Hilbert Theorem 90, there exists
$s\in Z_{T}(\theta' n)(\overline{k})$ such that $t(t^{-1})^\sigma=s(s^{-1})^\sigma$ for any $\sigma\in {\rm Gal}(\overline{k}/k)$.
This means that $s^{-1}t\in T(k)$ and we have
$${}^{s^{-1}t}g=s^{-1}gs=s^{-1}\theta'ns=\theta'n.$$
On the other hand $\theta'$ is conjugate to $\theta$ since $\theta'n={}^{y_{\alpha_1}\cdots y_{\alpha_r}}\theta'$.
It follows that they have to be conjugate by some $n'\in N$, hence $\theta'={}^{n'}\theta$.
This gives us the claim. The finiteness is then clear from the above description.
\end{proof}
\begin{remark} The proof of Lemma \ref{B-double} shows that Corollary 3 of \cite{L} holds for any local field $k$ of characteristic zero. 
\end{remark}
We are ready to prove Lemma \ref{P-double}.
\begin{proof}
The finiteness follows from the natural surjection $B(k)\backslash G(k)/H\lra P(k)\backslash G(k)/H$ and Lemma \ref{B-double}.

Henceforth we will make use of the mathematica code implemented by \cite{LV}.
By direct computation $n'$ runs over the set $R=\{1\}\cup \{n_\alpha\ |\ \alpha\in X\}$ where
$$
\begin{array}{rl}
X=&\{ 0000010, 0000110, 0000011, 0001110, 0000111, 0101110, 0011110, 0001111,\\
    &\ 1011110, 0111110, 0101111, 0011111, 1111110, 1011111, 0112110, 0111111,\\
    &\ 1112110, 1111111, 0112210, 0112111, 1122110, 1112210, 1112111, 0112211,\\
    &\ 1122210, 1122111, 1112211, 1123210, 1122211, 1223210, 1123211, 1223211
\}.
\end{array}
$$
In fact ${}^{n_{\alpha}}\theta=n_{\alpha}h_7(-1)n^{-1}_{\alpha}=h_{7}(-1)
h_{\alpha}((-1)^{\langle \beta_7,\alpha \rangle})\not= h_{7}(-1)$ is
equivalent to that $\langle \beta_7,\alpha \rangle$ is odd.

We shall discard extra elements among of $y_{\alpha_1}\ldots y_{\alpha_r}n', n'\in R$.
Recall that $E_6(k)\subset P(k)$ (resp. $H$) consists of roots generated by $\beta_1,\ldots,\beta_6$
(resp. $\gamma_1,\gamma_2=\beta_1,\ \gamma_3=\beta_3,\gamma_4=\beta_4,\gamma_5=\beta_5,\gamma_6=\beta_2,\beta_7$).

Assume $r=0$. We further assume that $P(k)n_\alpha H\not=P(k)H$ for $\alpha\in R$.
Then by direct computation, there exists $\beta\in \Phi$ so that $n_{\alpha}=n_\beta n n^{-1}_\beta$ and $n_{\beta}\in P(k)\cap H$
where $n$ is the element in the statement.
Hence we have $P(k)n_{\alpha}H=P(k)nH$.

Assume $r=1$. For $\alpha=\sum_{i=1}^7a_i\beta_i\in \Phi^+$, clearly $y_\alpha\in P(k)$ if $a_7=0$. Therefore the case $a_7>0$ will be
essential.
For each $n'\in R$ we compute the set $R_1(n')$
consisting of $\alpha$ so that $a_7>0$ and $\alpha({}^{n'}\theta)=-1$.
For example,
$$
\begin{array}{rl}
R_1(1)=&\{  0000011, 0000111, 0001111, 0101111, 0011111, 1011111, 0111111, 1111111,\\
       &\   0112111, 1112111, 0112211, 1122111, 1112211, 1122211, 1123211, 1223211
\}.
\end{array}
$$
By direct calculation for any $n'\in R$ and $\alpha\in R_1(n')$ we would check that there exists $g'\in G(k)$ such that
$$P(k)g'H=P(k)y_{\alpha}n'H\ {\rm and}\ Q_{g'}=Q_{y_7 n}.$$
Let us give a few examples.
For $n'=1$ and $\alpha=0000011$ we see that
$$g:=y_\alpha=n_6y_7n^{-1}_6\equiv y_7n_6\equiv y_7n_6n_7\equiv y_7n n_6\ {\rm mod}\ (P(k),H)$$
where $n_i=n_{\beta_i}$ and we use the relation $n=n_{\alpha}=n_{7}n_6n^{-1}_7$.
Put $g'=y_7n n_6$. Then one would be able to check $Q_{g'}=Q_{y_7n}$. 
The remaining cases would be done similarly. So it is omitted because it is a routine and lengthy. 
The case $r=2$ would be checked by using the calculation in case $r=1$. 
By direct calculation it is easy to check that 
the case $r\ge 3$ never happens because of the orthonormality for simple roots in question.  
\end{proof}

\subsection{An explicit structure of $Q_g$}\label{explicit}
By Lemma \ref{P-double} we may focus on the following four elements to consider $Q_g=g^{-1}P(k)g\cap H, g\in G(k)$.
The following table is made by Lawther. Here we put $H(\overline{k})=C_{G(\overline{k})}(\theta)$. 

\begin{table}[htbp]
\label{tab1}
\begin{center}
\begin{tabular}{|l|c|}
  \hline
  \text{$g\in G(k)$ } & $g^{-1}P(\overline{k})g\cap H(\overline{k})$ \\ \hline
  $g_0=1$ & $D_5 T_2 U_{11}$ \\
  $g_1=n$ & $A_5 A_1 T_1 U_{15}$ \\
  $g_2=y_{\beta_7} n$ &  $A_4 T_2 U_{21}$ \\
  $g_3=y_{\gamma_1} y_{\beta_7} n$  &   $B_3 A_1 T_1 U_{17}$ \\
\hline
\end{tabular}
\end{center}
\caption{}
\end{table}

For $i=0,1,2,3$, put $Q_i=g_i^{-1}Pg_i\cap H$. Let $Q_i=M_iN_i$ be the Levi decomposition and $T_i$ the maximal split torus in $M_i$. We now try to compute
$T_i, U_i$, and the values of the modulus character $\delta_{Q_i}$ (resp. the modulus character $\delta_{P(k)}$) on $T_i$
(resp. on $g_iT_ig^{-1}_i\subset P(k)$).
We first realize $G(k)$ in $GL_{56}(k)$ in terms of roots by using mathematica code implemented by \cite{LV}.
By using root groups we would know which entries of $P(k)=MU_{27}$ in $GL_{56}(k)$ are always zero (the number of such entries is
379).
This can be checked if we look $U^-_{27}=\{x_{\alpha}(c_\alpha)\ |\ \alpha\in\Phi^-, c_{\alpha}\in k\}$ because
the ($p$-adically) open subgroup $U^-_{27}P(k)$ is
Zariski dense in $E_7$ as an algebraic group. This gives rise to a naive criterion for $g\in G(k)$ to be an element of $P(k)$.
In what follows we denote by $|\ast|$ the normalized valuation of $k$ so that $|\varpi|=q^{-1}$ for a uniformizer of $k$ where
$q$ stands for the cardinality of the residue field of $k$.

\subsubsection{Case $Q_0$} In this case we have
$$T_0=\{h_{\gamma_1}(t_1)h_{\gamma_2}(t_2)h_{\gamma_3}(t_3)h_{\gamma_4}(t_4)h_{\gamma_5}(t_5)
h_{\gamma_6}(t_6)h_{\gamma_7}(t_7)\ |\ t_1,\ldots,t_7\in k^\times  \}$$
and $N_0=\langle x_{\alpha}(c_\alpha)\ |\ \alpha\in \Phi_0\rangle$
where
$$
\begin{array}{rl}
\Phi^0=&\{  0000001, 0112221, 1112221, 1122221, 1123221, 1123321, 1223221, 1223321,\\
       &\   1224321, 1234321,
2234321
\}.
\end{array}
$$
Then for $t=h_{\gamma_1}(t_1)h_{\gamma_2}(t_2)h_{\gamma_3}(t_3)h_{\gamma_4}(t_4)h_{\gamma_5}(t_5)
h_{\gamma_6}(t_6)h_{\gamma_7}(t_7)$ one has
$$\delta_{Q_0}(t)=|t_1|^{10}|t_2|^2\ {\rm and}\ \delta_{P(k)}(t)=|t_1t_7|^{18}.$$
Since $\delta_{P(k)}(t)=|\nu(t)|^{18}$ (see Section 6 of \cite{KY}), one concludes $\nu(t)=t_1t_2u$ for
some unit $u$ in $\mathcal{O}_k$.
In particular $\omega\circ \nu(t)=\omega(t_1t_2)$ for any unramified character $\omega$ of $k^\times$
where $\nu:P\lra GL_1$ is the similitude character.

It is easy to see that $G_1\cap T_0=\{ h_{\gamma_7}(t_7)\ t_7\in k^\times \}$ and
$G_2\cap T_0=\{ h_{\gamma_1}(t_1)\cdots h_{\gamma_6}(t_6)\ | \ t_1,\ldots,t_6\in k^\times  \}$.

\subsubsection{Case $Q_1$}
In this case we have
$$T_1=\{h_{\gamma_1}(t_1)h_{\gamma_2}(t_2)h_{\gamma_3}(t_3)h_{\gamma_4}(t_4)h_{\gamma_5}(t_5)
h_{\gamma_6}(t_6)h_{\gamma_7}(t_7)\ |\ t_1,\ldots,t_7\in k^\times  \}$$
and $N_1=\langle x_{\alpha}(c_\alpha)\ |\ \alpha\in \Phi_1\rangle$
where
$$
\begin{array}{rl}
\Phi^1=&\{  0000100, 0001100, 0101100,0011100, 1011100,  0111100,1111100 ,0112100, \\
       &\    1112100, 1122100, 1123321, 1223321, 1224321, 1234321, 2234321
\}.
\end{array}
$$

Then for $t=h_{\gamma_1}(t_1)h_{\gamma_2}(t_2)h_{\gamma_3}(t_3)h_{\gamma_4}(t_4)h_{\gamma_5}(t_5)
h_{\gamma_6}(t_6)h_{\gamma_7}(t_7)$ one has
$$\delta_{Q_1}(t)=|t_5|^{10}\ {\rm and}\ \delta_{P(k)}(g_1tg^{-1}_1)=|t_5|^{18}.$$
As seen before $\omega\circ \nu(g_1tg^{-1}_1)=\omega(t_5)$ for any unramified character $\omega$ of $k^\times$.
We also have $G_1\cap T_1=\{ h_{\gamma_7}(t_7)\ t_7\in k^\times \}$ and
$G_2\cap T_1=\{ h_{\gamma_1}(t_1)\cdots h_{\gamma_6}(t_6)\ | \ t_1,\ldots,t_6\in k^\times  \}$.

\subsubsection{Case $Q_2$}\label{Q2}
In this case we have
$$T_2=\{h_{\gamma_1}(t_5t_7)h_{\gamma_2}(t_2)h_{\gamma_3}(t_3)h_{\gamma_4}(t_4)h_{\gamma_5}(t_5)
h_{\gamma_6}(t_6)h_{\gamma_7}(t_7)\ |\ t_2,\ldots,t_7\in k^\times  \}$$
and $N_2=\langle x_{\alpha}(c_\alpha)\ |\ \alpha\in \Phi^2\rangle$
where
$$
\begin{array}{rl}
\Phi^2=&\{-0000001, 0000100, 0001100, 0101100, 0011100, 1011100, 0111100, 1111100, \\
       &\  0112100, 1112100, 1122100, 0112221, 1112221, 1122221, 1123221, 1223221, \\
       &\  1123321, 1223321, 1224321, 1234321, 2234321\}.
\end{array}
$$

Then for $t=h_{\gamma_1}(t_5t_7)h_{\gamma_2}(t_2)h_{\gamma_3}(t_3)h_{\gamma_4}(t_4)h_{\gamma_5}(t_5)
h_{\gamma_6}(t_6)h_{\gamma_7}(t_7)$ one has
$$\delta_{Q_2}(t)=|t_5|^{14}|t_7|^4\ {\rm and}\ \delta_{P(k)}(g_2tg^{-1}_2)=|t_5|^{18}.$$
As seen before, $\omega\circ \nu(g_2tg^{-1}_2)=\omega(t_5)$ for any unramified character $\omega$ of $k^\times$.
We also have $G_1\cap T_2=1$ and
$G_2\cap T_2=1$. 

The Levi of $Q_2$ is of type $A_4T_2$ and $A_4$ has simple roots  
$\beta_1,\ \gamma_3=\beta_3,\ \gamma_4=\beta_4,\ \gamma_6=\beta_2$. 
One can check that the centralizer $Z_{T_2}(A_4)=\{t\in T_2\ |\ tg=gt\ {\rm for\ any\ } g\in A_4\}$ is given by 
$$
T:=\{h_{T}(a):=h_{\gamma_1}(a^2)h_{\gamma_2}(a^2)h_{\gamma_3}(a^2)h_{\gamma_4}(a^2)h_{\gamma_5}(a)
h_{\gamma_6}(a)h_{\gamma_7}(a)\ |\ a\in k^\times \}\subset G_1G_2=K.
$$
We see that $GL_1$ is diagonally embedded in $K=G_1G_2$ via  
$\Delta:GL_1\lra T,\ a\mapsto h_{T}(a)$. 
Put $T':=\{h_{\gamma_1}(t_7)h_{\gamma_2}(t_2)h_{\gamma_3}(t_3)h_{\gamma_4}(t_4)
h_{\gamma_6}(t_6)h_{\gamma_7}(t_7)\}$ and $T'':=\{h_{\gamma_1}(t)h_{\gamma_7}(t)\ |\ t\in k^\times\}$. Then $T'A_4=T''\ltimes A_4$ makes up $GL_5$ and 
the projection $T'A_4\lra T''\simeq GL_1$ corresponds to the determinant. 

\subsubsection{Case $Q_3$}
The situation is a little bit more complicated than other cases.
Let us first observe that
$$P(k)\cap g_3Hg^{-1}_3 =P(k)\cap g_3 C_{G(k)}(\theta) g^{-1}_3 = C_{P(k)}(g_3 \theta g^{-1}_3) = C_{P(k)}(g'),$$
where $g' = h_{\beta_6}(-1) n_{\beta_7} n_{\gamma_1}$.
One can easily extend Theorem 6 of \cite{L} to the Siegel parabolic subgroup $P$ and then
we get ${\rm dim}N_3=17$. On the other hand one can consider the unipotent subgroup $U_{17}$ directly in $P \cap gHg^{-1}$ as follows.
For the 16 of the 17 root groups in $U_{17}$, there is then a 1-dimensional unipotent group diagonally
embedded in the product of the two root groups, of the form
$$\{ x_\alpha(t) g' x_\alpha(t) {g'}^{-1} : t \in k \} = \{ x_\alpha(t)  x_{g'(\alpha)}(\pm t) : t \in k \},
$$
 where the sign in the second term is determined by the structure constants.
The 17th root subgroup is simply the root subgroup corresponding to the highest root 2234321.
The 16 pairs of positive roots $\alpha$, $g'(\alpha)$ interchanged by $g'$ are as follows:

\medskip

\begin{center}
\begin{tabular}{|c|c|c|c|c|c|c|c|}
  \hline
  $\alpha$ & $g'(\alpha)$  &  $\alpha$ &  $g'(\alpha)$ & $\alpha$ &  $g'(\alpha)$  & $\alpha$ &  $g'(\alpha)$  \\ \hline
 1000000  & 1112221 & 1011110 & 1011111 & 1010000  & 1122221 & 1111110  & 1111111 \\
 1011000  & 1123221 & 1112110 & 1112111 & 1011100  & 1123321 &  1122110 & 1122111 \\
1111000  & 1223221 & 1112210 & 1112211 & 1111100  & 1223321 &  1122210 & 1122211\\
1112100  & 1224321 &1123210 & 1123211 & 1122100 & 1234321 &1223210 & 1223211 \\
\hline
\end{tabular}
\end{center}

\medskip

By matching of the dimension we may have $g^{-1}_3U_{17}g_3=N_3$.
On the other hand we have
$$T_3=\{h_{\gamma_1}(t_5t_7)h_{\gamma_2}(t^2_5)h_{\gamma_3}(t_3)h_{\gamma_4}(t_4)h_{\gamma_5}(t_5)
h_{\gamma_6}(t_6)h_{\gamma_7}(t_7)\ |\ t_3,\ldots,t_7\in k^\times  \}.$$
Then for $t=h_{\gamma_1}(t_5t_7)h_{\gamma_2}(t^2_5)h_{\gamma_3}(t_3)h_{\gamma_4}(t_4)h_{\gamma_5}(t_5)
h_{\gamma_6}(t_6)h_{\gamma_7}(t_7)$ one has
$$\delta_{Q_3}(t)=|t_5|^{18}\ {\rm and}\ \delta_{P(k)}(g_3tg^{-1}_3)=|t_5|^{18}.$$
As seen before, $\omega\circ \nu(g_3tg^{-1}_3)=\omega(t_5)$ for any unramified character $\omega$ of $k^\times$.
We also have $G_1\cap T_3=\{ h_{\gamma_7}(t_7)\ t_7\in k^\times \}$ and
We also have $G_1\cap T_2=1$ and
$G_2\cap T_2=1$.
Finally we remark that $G_1=SL_2$ is common factor of $G_1$ and $G_2$, hence there exists a
2 to 1 homomorphism
\begin{equation}\label{common}
\Delta:SL_2\lra  G_1\times G_2 \lra H
\end{equation} onto the image. Let $\iota:SL_2\lra SL_2$ be the isomorphism defined by
$\begin{pmatrix}
\alpha & \beta \\
\gamma & \delta
\end{pmatrix}
\mapsto
\begin{pmatrix}
\alpha & -\beta \\
-\gamma & \delta
\end{pmatrix}
$.
The image of $\Delta$ is naturally isomorphic to
\begin{equation}\label{image}
\{(\gamma,\iota(\gamma))\ |\ \gamma\in SL_2\}/\{\pm(I_2,I_2)\}.
\end{equation}

\section{Computation of Satake parameters}\label{csp}

In this section, we prove Proposition \ref{key} below, which is a key to the proof of Theorem \ref{main}. It is an analogue of Proposition 3.1 of \cite{Ik1}.
Recall $G_1=SL_2(\Q_p)$ and $G_2=Spin(12)(\Q_p)$.
Let $\pi_2'$ be an unramified principal series representation of $G'=GSpin(12)(\Q_p)$ with the trivial central character. We compute Satake parameters of $\pi_2'$.

Since the group $G_2$ appears as a subgroup of $E_7$, we need to
consider the restriction $\pi_2=\pi_2'|_{Spin(12)}$.

Since ${}^L GSpin(12)=GSO(12,\Bbb C)$, the Satake parameter of $\pi_2'$ is given by
$$(b_1,b_2,\ldots,b_6,b_6^{-1}b_0,\ldots,b_2^{-1}b_0,b_1^{-1}b_0)\in GSO(12,\Bbb C)
$$
for some $b_1,...,b_6\in \Bbb C^\times$, and $b_0=\omega_{\pi_2'}(p)$. Since the central character is trivial, $b_0=1$.

Let $\pi_i$ be an unramified principal series representation of $G_i$ for $i=1,2$. Then
$\pi_i={\rm Ind}^{G_i}_{B_i}\chi_i$, where $B_1, B_2$ are the standard Borel subgroups of
$G_1, G_2$, resp. and $\chi_i: B_i\lra\C^\times$ is an unramified character.
 The modulus character of each $B_i$ is given by
$$\delta_{B_1}(h_{\gamma_7}(t_7))=|t_7|^2,\ \delta_{B_2}(h_{\gamma_1}(t_1)\cdots h_{\gamma_6}(t_6))=\prod^6_{i=1}|t_i|^2.
$$
Here ``Ind" stands for the normalized induction and we will denote by ``c-Ind" the compact normalized induction.

Let $\{\beta^{\pm 1}\}$ be the Satake parameters of $\pi_1$. Then
we have
\begin{equation}\label{satake}
\chi_1(h_{\gamma_7}(p^{-1}))=\beta^2.
\end{equation}

Also we have
\begin{equation}\label{satake1}
\begin{array}{l}
\chi_2(h_{\gamma_i}(p^{-1}))=\ds\frac{b_i}{b_{i+1}}, 1\le i\le 4,\
\chi_2(h_{\gamma_6}(p^{-1}))=\frac{b_5}{b_6},\
\chi_2(h_{\gamma_5}(p^{-1}))=b_5b_6.
\end{array}
\end{equation}
Recall that $H=C_{G}(\theta)\simeq (G_1\times G_2)/Z$ where $Z\simeq \{\pm 1\}$ is diagonally embedded in both centers.
Let $\phi: G_1\times G_2\longrightarrow H$ be the isogeny. As seen before  
$\phi(G_1(\Q_p)\times G_2(\Q_p))$ is a finite index subgroup of 
$H(\Q_p)$. 
Let $B_H$ be a Borel subgroup of $H$. Let $\chi$ be a character of $B_H$ and let $\pi(\chi)$ be the spherical subquotient of 
${\rm Ind}_{B_H}^H \chi$. Let $\widetilde\chi=\chi\circ\phi$ be a character of $B_1\times B_2$ and let $\pi(\widetilde\chi)$ be the spherical subquotient of ${\rm Ind}_{B_1\times B_2}^{G_1\times G_2} \widetilde\chi$. 
Then we have a surjective map between unramified $L$-packets:
$$\Pi(H(\Q_p))\lra \Pi(G_1(\Q_p)\times G_2(\Q_p)),\quad \pi(\chi)\mapsto \pi(\widetilde{\chi}).
$$ 
Given $\chi_1,\chi_2$, unramified characters of $B_1, B_2$, resp., there exist finitely many $\chi$ of $B_H$ such that $\chi_1\otimes\chi_2^{-1}=\widetilde{\chi}$. 
Let $\pi_H=\pi(\chi)$ for any such $\chi$.
 Then 
$\pi_H$ is a subquotient of ${\rm Ind}_{\phi(G_1(\Q_p)\times G_2(\Q_p))}^{H(\Q_p)} \, \pi_1\otimes\widetilde{\pi_2}$, and if $\pi_1\otimes\widetilde{\pi_2}$ is unitary, then $\pi_H$ is unitary (Lemma 2.3 of \cite{Li}). 
We call $\pi_H$ a lift of $\pi_1\otimes\widetilde{\pi_2}$ by abuse of notation. Note that for 
$t\in T=\{\prod_{i=1}^7h_i(t_i)\ |\ t_i\in \Q^\times_p\}$, $\pi_H(t)$ acts by 
$(\pi_1\otimes\widetilde{\pi_2})(t)=\chi_1(t_7)\chi^{-1}_2(\prod_{i=1}^6h_i(t_i))$.

Let $\omega: \Q^\times_p \rightarrow \Bbb C^\times$ be an unramified unitary character and let $\alpha=\omega(p^{-1})$.

\begin{prop} \label{key}
Assume that $\pi_1\otimes\widetilde{\pi_2}$ is unitary. 
If ${\rm Hom}_{H}({\rm Ind}_P^G \, (\omega^{-2}\circ \nu)|_{H}, \pi_H)\ne 0$ for some lift $\pi_H$ of $\pi_1\otimes\widetilde{\pi_2}$, 
then as a multiset,
$\{{b_1}^{\pm 1},..., {b_6}^{\pm 1}\}$ is equal to one of the followings:
\begin{eqnarray*}
&&(I):\   \{ \varepsilon(\beta\alpha)^{\pm 1}, \ \varepsilon(\beta\alpha^{-1})^{\pm 1}, b^{\pm 1},\ (bp)^{\pm 1},\ 
(bp^2)^{\pm 1},\ (bp^3)^{\pm 1}\},
\quad {\rm or}   \\
&&(II):\ \{ \varepsilon(\beta\alpha)^{\pm 1}, \ \varepsilon(\beta\alpha^{-1})^{\pm 1}, 
\varepsilon(\beta\alpha^{-1}p)^{\pm 1},\ \varepsilon(\beta\alpha^{-1}p^2)^{\pm 1},\ 
\varepsilon(\beta\alpha^{-1}p^3)^{\pm 1},\ 
\varepsilon(\beta\alpha^{-1}p^4)^{\pm 1}\}
\end{eqnarray*}
where $\varepsilon\in \{\pm 1\}$, and $b\in\Bbb C^{\times}$.
\end{prop}

\begin{proof}
By Lemma \ref{P-double} one can take the representatives
$\{h_n\}^r_{n=1}$ of $P(k)\bs G(k)/H$ so that $Q_{h_n}\in\{Q_i\ |\ i=0,1,2,3  \}$.
Then in the category of Grothendieck group of admissible representations we have
$${\rm Ind}_P^G \, (\omega^{-1}\circ \nu)|_{H}=\sum_{n=1}^r c{\rm -Ind}^{H}_{Q_{h_n}}\omega_n \delta^{-\frac{1}{2}}_{Q_{h_n}} $$
where $\omega_n(g)=\delta^{\frac{1}{2}}_{P(k)}(h_ngh^{-1}_n)\omega^{-2}\circ \nu(h_ngh^{-1}_n)$ for $g\in Q_{h_n}$.
Put $\omega_n=\omega_i$ if $Q_{h_n}=Q_i$.
Then by assumption there exists $i$ $(0\le i\le 3)$ such that $Q_{h_n}=Q_i$ and
$$
\begin{array}{rl}
0\not=&{\rm Hom}_{H}(c{\rm -Ind}^{H}_{Q_{i}}\omega_i \delta^{-\frac{1}{2}}_{Q_{i}},\pi_H)\\
=&{\rm Hom}_{H}(\widetilde{\pi_H},{\rm Ind}^{H}_{Q_{i}}\omega^{-1}_i \delta^{\frac{1}{2}}_{Q_{i}})\\
=&{\rm Hom}_{Q_i}(\widetilde{\pi_H}|_{Q_i},\omega^{-1}_i) \ \mbox{(by Frobenius reciprocity)} 
\end{array}
$$
In the case of $Q_0$, we observe the action of $h_{\gamma_7}(p^{-1})\in Q_0$ on both spaces.
Then one has
$\beta^2=p^{-9}$ which contradicts to the unitarity of $\pi_1$.
Similarly we observe the action of $h_{\gamma_7}(p^{-1})$ for $Q_1$. Then it gives a contradiction that 
$p\beta^2=1$. 

In the case of $Q_2$, applying (\ref{satake}) and (\ref{satake1}) to the following elements
$$h_{\gamma_2}(p^{-1}),\ 
h_{\gamma_3}(p^{-1}),\ h_{\gamma_4}(p^{-1}),\ h_{\gamma_6}(p^{-1}),
\ h_{\gamma_1}(p^{-1})h_{\gamma_5}(p^{-1}) ,\ h_{\gamma_1}(p^{-1})h_{\gamma_7}(p^{-1})\in T_2
$$
respectively, we have
\begin{equation}\label{rel02}
p\ds\frac{b_2}{b_3}=1,\
p\ds\frac{b_3}{b_4}=1,\
p\ds\frac{b_4}{b_5}=1,\
p\ds\frac{b_5}{b_6}=1,\
(p\frac{b_1}{b_2})(pb_5b_6)=p^9\alpha^2_p,\
(p^{-1}\beta^{-2})\left(p\frac{b_1}{b_2}\right)=1.
\end{equation}
From this, we obtain the Satake parameters 
$$(II): \{ \varepsilon(\beta\alpha)^{\pm 1}, \ \varepsilon(\beta\alpha^{-1})^{\pm 1}, 
\varepsilon(\beta\alpha^{-1}p)^{\pm 1},\ \varepsilon(\beta\alpha^{-1}p^2)^{\pm 1},\ 
\varepsilon(\beta\alpha^{-1}p^3)^{\pm 1},\ 
\varepsilon(\beta\alpha^{-1}p^4)^{\pm 1}\}
$$
for some $\varepsilon\in \{\pm 1\}$. 

Finally we consider the case of $Q_3$. 
For
$t=h_{\gamma_1}(t_5t_7)h_{\gamma_2}(t^2_5)h_{\gamma_3}(t_3)h_{\gamma_4}(t_4)h_{\gamma_5}(t_5)
h_{\gamma_6}(t_6)h_{\gamma_7}(t_7)$,
 we see $\omega^{-1}_3(t)=\omega^{2}(t_5)|t_5|^9\delta_{B_2}^{\frac 12}(t)=\prod_{i=1}^6 |t_i|$.
In this case, applying (\ref{satake}) and (\ref{satake1}) to the following elements
$$
h_{\gamma_3}(p^{-1}),\ h_{\gamma_4}(p^{-1}),\ h_{\gamma_6}(p^{-1}),\ h_{\gamma_1}(p^{-1})h_{\gamma_2}(p^{-2})
h_{\gamma_5}(p^{-1})h_{\gamma_6}(p),\ h_{\gamma_1}(p^{-1})h_{\gamma_7}(p^{-1})\in T_3
$$
respectively, we have
\begin{equation}\label{rel2}
p\ds\frac{b_3}{b_4}=1,\
p\ds\frac{b_4}{b_5}=1,\
p\ds\frac{b_5}{b_6}=1,\
p^3\ds\frac{b_1b_2b_6^2}{b_3^2}=p^9\alpha^2,\
(p^{-1}\beta^{-2})\left(p\frac{b_1}{b_2}\right)=1.
\end{equation}
From the first four equalities, we have $b_1b_2=\alpha^2.$ From the last equation, $\ds\frac {b_1}{b_2}=\beta^2$.
Hence $b_1^2=(\alpha\beta)^2$. Hence $b_1=\varepsilon \alpha\beta$ and $b_2=\varepsilon\ds\frac {\alpha}{\beta}$, where 
$\varepsilon=\pm 1$.

It follows from (\ref{rel2}) that
\begin{equation}\label{rel3}
b_4=pb_3,\
b_5=p^2b_3,\
b_6=p^3b_3,
\end{equation}
where $b_3\in\Bbb C^\times$. Hence the Satake parameters of $\pi_2'$ are

$$(I): \{\varepsilon (\alpha\beta)^{\pm1},\ \varepsilon(\alpha{\beta}^{-1})^{\pm 1},\ (bp^3)^{\pm 1},\ (bp^2)^{\pm 1},\ (bp)^{\pm 1},\ b^{\pm 1}\},
$$
where $\varepsilon\in\{\pm 1\}$ and $b\in\Bbb C^\times$.
\end{proof}

\section{Proof of Theorem \ref{main}} \label{main-section}
Let $\mathcal H(G_i(\Bbb A_f))$ ($i=1,2$) be the Hecke algebra for the finite adele group $G_i(\Bbb A_f)$. Then $\mathcal H(G_1(\Bbb A_f))\cdot h$ and $\mathcal H(G_2(\Bbb A_f))\cdot \mathcal{F}_{f,h}$ are the finite part of the cuspidal automorphic representations of $G_1(\Bbb A)$ and $G_2(\Bbb A)$ generated by $h$ and $\mathcal{F}_{f,h}$, resp. Here $\mathcal H(G_1(\Bbb A_f))\cdot h$ is an irreducible representation of $G_1(\Bbb A_f)$.
Let $\pi_1$ be the $p$-component of $\mathcal H(G_1(\Bbb A_f))\cdot h$. Then $\pi_1$ is an unramified principal series with the Satake parameter
$\{\beta_p^{\pm 1}\}$. On the other hand, since $\mathcal{F}_{f,h}(Z)$ is a cusp form, the representation $\mathcal H(G'(\Bbb A_f))\cdot \mathcal{F}_{f,h}$ of
$G'(\Bbb A_f)$ is unitary and of finite length, where $G'=GSpin(2,10)$. We consider the restriction to $G_2(\Bbb A_f)$, and let $\pi_2$ be the
$p$-component of some irreducible direct summand of that restriction. Then $\pi_2$ is also an unramified principal series.

Note that $\det({\rm Im} Z)^{-10} dZ$ is the invariant measure on $G'(\Z)\backslash \frak T_2$. Then if $\mathcal{F}_{f,h}\ne 0$,
$$\int_{G'(\Z)\backslash \frak T_2} \int_{SL_2(\Z)\backslash \Bbb H}  \overline{F\begin{pmatrix} Z&0\\0&\tau\end{pmatrix}} h(\tau) 
\mathcal{F}_{f,h}(Z) ({\rm Im} \, \tau)^{2k+6} \det({\rm Im}\, Z)^{2k-2}\, dZd\tau=\langle\mathcal{F}_{f,h},\mathcal{F}_{f,h}\rangle\ne 0.
$$
It follows from this that for each prime $p$, 
$$
\begin{array}{rl}
0\not=&{\rm Hom}_{H(\Bbb Q_p)}(\text{Ind}_{P(\Bbb Q_p)}^{G(\Bbb Q_p)} \,
 (\omega^{-2}_p\circ \det)|_{\phi(G_1(\Q_p)\times G_2(\Q_p))}, \pi_1\otimes \widetilde{\pi_2})\\
=&{\rm Hom}_{H(\Bbb Q_p)}(\text{Ind}_{P(\Bbb Q_p)}^{G(\Bbb Q_p)} \,
 (\omega^{-2}_p\circ \det)|_{H(\Q_p)}, {\rm Ind}^{H(\Q_p)}_{\phi(G_1(\Q_p)\times G_2(\Q_p))}\pi_1\otimes \widetilde{\pi_2})
\end{array}
$$
and this implies 
$${\rm Hom}_{H(\Bbb Q_p)}(\text{Ind}_{P(\Bbb Q_p)}^{G(\Bbb Q_p)} \, (\omega^{-2}_p\circ \det)|_{H(\Bbb Q_p)}, \pi_H)\ne 0
$$
for some lift $\pi_H$ to $H(\Q_p)$ of $\pi_1\otimes \widetilde{\pi_2}$ defined as in Proposition \ref{key}, 
where $\omega_p: \Bbb Q_p^\times\lra \Bbb C^\times$ is the unramified character determined by $\omega_p(p^{-1})=\alpha_p$. 
By Proposition \ref{key}, any irreducible component of $\mathcal H(G'(\Bbb A_f))\cdot \mathcal{F}_{f,h}$ has the Satake $p$-parameter
\begin{eqnarray*}
&&(I)_p:\ \{\varepsilon_p(\beta_p\alpha_p)^{\pm 1},\ \varepsilon_p(\beta_p\alpha_p^{-1})^{\pm 1},\  (b_p p^3)^{\pm 1},\ (b_p p^2)^{\pm 1},\ (b_p p)^{\pm 1},\ b_p^{\pm 1}\},
\quad {\rm or}   \\
&&(II)_p:\ \{ \varepsilon_p(\beta_p\alpha_p)^{\pm 1}, \ \varepsilon_p(\beta_p\alpha_p^{-1})^{\pm 1}, 
\varepsilon_p(\beta_p\alpha_p^{-1}p)^{\pm 1},\ \varepsilon_p(\beta_p\alpha_p^{-1}p^2)^{\pm 1},\ 
\varepsilon_p(\beta_p\alpha_p^{-1}p^3)^{\pm 1},\ 
\varepsilon_p(\beta_p\alpha_p^{-1}p^4)^{\pm 1}\},
\end{eqnarray*}
where $\varepsilon_p=\pm 1$ and $b_p\in\Bbb C^\times$. 

Now we assume the Langlands functorial transfer of automorphic representations of $PGSpin(2,10)(\Bbb A)$ to $GL_{12}(\Bbb A)$ as in the introduction.

Let $\Pi_{f,h}$ be an irreducible component of the cuspidal representation of $G'(\Bbb A))$ generated by $\mathcal{F}_{f,h}$. Then it is unramified at every prime $p$.
Let $\Pi$ be the transfer of $\Pi_{f,h}$ to $GL_{12}(\Bbb A)$. Then $\Pi$ is unramified at all $p$ by the property of Langlands functoriality. 
By the classification of automorphic representations of $GL_N$ \cite{JS}, $\Pi$ is the Langlands' quotient of 
$$\sigma_1|\det|^{r_1}\boxplus\cdots\boxplus\sigma_k|\det|^{r_k}\boxplus \sigma_{k+1}\boxplus\cdots\boxplus \sigma_{k+l}\boxplus \tilde\sigma_k|\det|^{-r_k}\boxplus\cdots\boxplus\tilde\sigma_1|\det|^{-r_1},
$$
where $r_1\geq r_2\geq\cdots\geq r_k>0$, and $\sigma_1,...,\sigma_{k+l}$ are unitary (irreducible) cuspidal representations of $GL_{n_i}(\Bbb A)$. Note also that if $(c_{1p},...,c_{mp})$ are Satake parameters of a cuspidal representation $\pi$ of $GL_m(\Bbb A)$, 
$p^{-\frac 12}<|c_{ip}|<p^{\frac 12}$ for each $i$. Hence
by comparing the Satake parameters, 
the Satake parameters should be either $(I)_p$ for all $p$, or $(II)_p$ for all $p$. 

Suppose the Satake parameters are $(II)_p$ for all $p$. Then
$\Pi$ is the Langlands' quotient of 
$$\Pi=\Pi_1\boxplus (\chi|\cdot |)^{\pm 1}\boxplus (\chi|\cdot |^2)^{\pm 1}\boxplus (\chi|\cdot |^3)^{\pm 1}
\boxplus (\chi|\cdot |^4)^{\pm 1},
$$ where $\chi:\Q\backslash\A^\times_\Q\lra \C^\times$ is a unitary 
idele class character, and $\Pi_1$ is an automorphic representation of $GL_4(\Bbb A)$ whose Satake parameters are
$\{ \varepsilon_p(\beta_p\alpha_p)^{\pm 1}, \ \varepsilon_p(\beta_p\alpha_p^{-1})^{\pm 1}\}$ at each $p$. The automorphy of $\Pi_1$ is explained as follows: We can see easily that $\wedge^2\Pi_1={\rm Sym}^2(\pi_f)\oplus {\rm Sym}^2(\pi_h)$. It is an automorphic representation of $GL_6(\Bbb A)$. Now the exterior square $\wedge^2: GL_4(\Bbb C)\longrightarrow GL_6(\Bbb C)$ is the composition of $\hat\phi: GL_4(\Bbb C)\longrightarrow GSO_6(\Bbb C)$ and $\iota: GSO_6(\Bbb C)\longrightarrow GL_6(\Bbb C)$, where $\iota$ is the embedding, and $\phi: GSpin_6\longrightarrow GL_4$ is the double covering map \cite{AS}. Hence the exterior square transfer is the composition of transfers from 
$GL_4(\Bbb A)\longrightarrow GSpin_6(\Bbb A)$ and $GSpin_6(\Bbb A)\longrightarrow GL_6(\Bbb A)$. Since the central character of $\Pi_1$ is trivial, it is a representation of $PGL_4\simeq PGSpin_6=PGSO_6$. Hence for representations with the trivial central character, the exterior square transfer is the transfer $PGSO_6(\Bbb A)\longrightarrow GL_6(\Bbb A)$. 
Now by the result of Arthur \cite{Art}, since $\wedge^2\Pi_1$ is automorphic, $\Pi_1$ is an automorphic representation of 
$PGSO_6(\Bbb A)\simeq PGL_4(\Bbb A)$. 

Since $\chi$ is the global unramified character, one must have $\chi=1$, i.e.,
$\alpha_p=\beta_p$ and $\varepsilon_p=1$ for all $p$. Since $f$ and $h$ have different weights, they can never be equal. Contradiction. 

Hence the Satake parameters should be $(I)_p$ for all $p$. Now we recall the classification of spherical unitary representations of $GL_N(\Bbb Q_p)$ \cite{T}: For an unramified unitary character $\chi$, let $\chi(\det_n)$ be the representation $g\longmapsto \chi(\det_n(g))$ of $GL_n(\Bbb Q_p)$. Let $\pi(\chi(\det_n),\alpha)$ be the representation of $GL_{2n}(\Bbb Q_p)$ induced by $\chi(\det_n)|\det|^{\alpha}\otimes \chi(\det_n)|\det|^{-\alpha}$, where $0<\alpha<\frac 12$. Then any spherical unitary representation of $GL_N(\Bbb Q_p)$ is induced by
$$\chi_1({\det}_{n_1})\otimes\cdots \otimes\chi_q({\det}_{n_q})\otimes \pi(\mu_1({\det}_{m_1}),\alpha_1)\otimes\cdots\otimes \pi(\mu_r({\det}_{m_r}),\alpha_r),
$$
where $n_1+\cdots+n_q+2(m_1+\cdots+m_r)=N$, $0<\alpha_1,...,\alpha_r<\frac 12$, and $\chi_1,...,\chi_q,\mu_1,...,\mu_r$ are unramified unitary characters. Hence by comparing the Satake parameters, we can see that $|b_p|=1$ for all $p$. Since
$\Pi$ is unramified everywhere, we conclude that $b_p=1$. Hence $\Pi$ is the Langlands' quotient of 
$\Pi=\Pi_1\boxplus 1\boxplus 1\boxplus |\ |^{\pm 1}\boxplus |\ |^{\pm 2}\boxplus |\ |^{\pm 3}$. 
Since $\wedge^2\Pi_1={\rm Sym}^2(\pi_f)\oplus {\rm Sym}^2(\pi_h)$, by \cite{AR}, $\Pi_1$ is of the form $\Pi_1=\sigma_1\boxtimes\sigma_2$ for $\sigma_1, \sigma_2$, cuspidal representations of $GL_2(\Bbb A)$. Since $\wedge^2(\sigma_1\boxtimes\sigma_2)=Ad(\sigma_1)\otimes\omega_{\sigma_1}\omega_{\sigma_2}\boxplus Ad(\sigma_2)\otimes\omega_{\sigma_1}\omega_{\sigma_2}$, $\omega_{\sigma_1}\omega_{\sigma_2}=1$, $Ad(\sigma_1)=Ad(\pi_f)$ and $Ad(\sigma_2)=Ad(\pi_h)$. By \cite{Ram},
$\sigma_1=\pi_f\otimes\chi_1$ and $\sigma_2=\pi_h\otimes\chi_2$ for some characters $\chi_1,\chi_2$. Hence $\Pi_1=(\pi_f\boxtimes\pi_h)\otimes \chi_1\chi_2$. However $\chi_1\chi_2$ has to be 
trivial because $\Pi_1$ is unramified everywhere. Therefore, $\Pi_1=\pi_f\boxtimes\pi_h$, and
$\epsilon_p=1$ for all $p$. This shows that $\Pi=(\pi_f\boxtimes\pi_h)\boxplus 1_{GL_7}\boxplus 1$, where $1_{GL_7}$ is the trivial representation of $GL_7(\Bbb A)$.

The Satake parameters at $p$ behave uniformly and it follows from this that $\mathcal H(G'(\Bbb A_f))\cdot \mathcal{F}_{f,h}$ is isotypic. Since it is generated by the class one vector 
$\mathcal{F}_{f,h}$, it is irreducible. It follows that $\mathcal{F}_{f,h}$ is a Hecke eigenform and gives rise to a cuspidal representation $\Pi_{f,h}$ of $G'(\Bbb A)$. We also showed that the degree 12 standard $L$-function is

$$L(s,\Pi_{f,h})=L(s,\pi_f\times \pi_h)\zeta(s)^2 \zeta(s\pm 1)\zeta(s\pm 2)\zeta(s\pm 3),
$$  
where the first $L$-function is the Rankin-Selberg $L$-function.

\section{Remark on non-vanishing hypothesis}

Recall
$$\mathcal{F}_{f,h}(Z)=\int_{SL_2(\Z)\backslash \Bbb H} F_f\begin{pmatrix} Z&0\\0&\tau\end{pmatrix} \overline{h(\tau)}
({\rm Im} \tau)^{2k+6}\, d\tau.
$$

We consider the nonvanishing question of $\mathcal{F}_{f,h}$.
We have two  Fourier-Jacobi expansions of $F_f$;

\begin{equation}\label{FJ}
F_f\begin{pmatrix} Z&w\\ {}^t w&\tau\end{pmatrix}=\sum_{m=1}^\infty \phi_m(Z,w) e^{2\pi i m\tau}=\sum_{S} \mathcal F_S(\tau,w) e^{2\pi i Tr(Z S)},
\end{equation}
where $\phi_m$ is a Jacobi cusp form of weight $2k+8$ of index $m$ as in \cite{E}. In the second sum,
$S\in \frak J_2^+(\Z)$ and $\mathcal F_S$ is a Fourier-Jacobi coefficient of index $S$ as in \cite{KY}. Here
$$\mathcal F_S(\tau,w)=\sum_{\lambda\in\Lambda} \theta_{[\lambda]}(S;\tau,w)\mathcal F_{S,\lambda}(\tau),
$$
where $\theta_{[\lambda]}(S;\tau,w)$ is a theta series and $\mathcal F_{S,\lambda}(\tau)$ is a vector-valued modular form, which is obtained from the compatible family of Eisenstein series.

\begin{lemma}  We have the estimates:
 \begin{eqnarray*} |\phi_m(Z,0)| &\ll &  \det(Y)^{-(2k+8)} m^{2k+8}, \quad Y={\rm Im}(Z),\\
|\mathcal F_S(\tau,0)|&\ll& y^{-(2k+8)}  Tr(S)^{2(2k+8)},\quad y={\rm Im}(\tau).
\end{eqnarray*}
\end{lemma}
\begin{proof} From the first expansion in (\ref{FJ}), for any $y>0$,
$$\phi_m(Z,0) e^{-2\pi m y}=\int_0^1 F_f\begin{pmatrix} Z&0\\ 0&\tau\end{pmatrix}\, e^{-2\pi i m x}\, dx.
$$
Here $\left|F_f\begin{pmatrix} Z&0\\ 0&\tau\end{pmatrix}\right|\ll (\det({\rm Im}(Z)) {\rm Im}(\tau))^{-(2k+8)}$. 
Set $y=\frac 1m$. Then 
$$|\phi_m(Z,0)| \ll \det(Y)^{-(2k+8)} m^{2k+8}. 
$$

From the second expansion in (\ref{FJ}),
$$\mathcal F_S(\tau,0) e^{-2\pi Tr(Y S)}=\int_{X} F_f\begin{pmatrix} Z&0\\ 0&\tau\end{pmatrix}\, e^{-2\pi i Tr(X S)}\, dX,
$$
where the integral is over $\frak T_2(\Bbb R)/\frak T_2(\Bbb Z)$. 
Set $Y=\frac 1{Tr(S)} I_2$. Then 
$$|\mathcal F_S(\tau,0)|\ll  y^{-(2k+8)} Tr(S)^{2(2k+8)}. 
$$
\end{proof}

Consider the first Fourier-Jacobi expansion. We have
$$\sum_{m=1}^\infty |\phi_m(Z,0) e^{2\pi i m\tau}|\leq \sum_{m=1}^\infty m^{2k+8} \det(Y)^{-(2k+8)} e^{-2\pi m y}\ll e^{-2\pi y} \det(Y)^{-(2k+8)}.
$$
Since $|h(\tau)|\ll e^{-2\pi y}$, 
$$\int_{SL_2(\Z)\backslash \Bbb H}  F_f\begin{pmatrix} Z&0\\0&\tau\end{pmatrix} \overline{h(\tau)}
y^{2k+6}\, dx dy
$$
converges absolutely.
Hence we can interchange the sum and integral. So

$$\mathcal{F}_{f,h}(Z)=\sum_{m=1}^\infty \phi_m(Z,0) \int_{SL_2(\Z)\backslash \Bbb H} e^{2\pi i m \tau} \overline{h(\tau)} {\rm Im}(\tau)^{2k+6}\, d\tau.
$$
Here $\phi_m(Z,0)$ is a linear combination of cusp forms on $GSpin(2,10)$.  
So it is unlikely that $\mathcal F_{f,h}$ is identically zero.

Next consider the second Fourier-Jacobi expansion of $F_f$. From the above lemma,
$$\sum_{S} |\mathcal F_S(\tau,0) e^{2\pi i Tr(Z S)}|\leq \sum_{S}  y^{-(2k+8)}  Tr(S)^{2(2k+8)} e^{-2\pi Tr(Y S)}\ll y^{-(2k+8)}.
$$
Since $|h(\tau)|\ll e^{-2\pi y}$, 
$$\int_{SL_2(\Z)\backslash \Bbb H}  F_f\begin{pmatrix} Z&0\\0&\tau\end{pmatrix} \overline{h(\tau)}
y^{2k+6}\, dx dy
$$
converges absolutely.
Hence we can interchange the sum and integral. 
So
$$\mathcal F_{f,h}(Z)=\sum_{S} A_S \, e^{2\pi i Tr(Z S)},
$$
where
$$A_S=\int_{SL_2(\Z)\backslash \Bbb H} \mathcal F_S(\tau,0) \overline{h(\omega)} {\rm Im}(\tau)^{2k+6}\, d\tau
=\sum_{\lambda\in\Lambda} \int_{SL_2(\Z)\backslash \Bbb H} \theta_{[\lambda]}(S;\tau,0)\mathcal F_{S,\lambda}(\tau)
\overline{h(\tau)} {\rm Im}(\tau)^{2k+6}\, d\tau.
$$

Here $\mathcal F_S(\tau,0)$ is a modular form of weight $2k+8$. Hence $A_S$ is the Petersson inner product of $\mathcal F_S(\tau,0)$ and $h$.
This expression shows that it is very likely that $\mathcal F_{f,h}$ is not identically zero.

\end{document}